\documentclass[12pt]{article}
\makeatletter

\usepackage{amssymb,amsmath,amsthm,latexsym}
\title{}
\author{}
\date{}

\begin{document}

\title{The isotopism problem of a class of 6-dimensional rank 2 semifields and its solution}
\author{M. Lavrauw, G. Marino, O. Polverino, R. Trombetti}

\maketitle

\newtheorem{theorem}{Theorem}[section]
\newtheorem{lemma}[theorem]{Lemma}
\newtheorem{conj}[theorem]{Conjecture}
\newtheorem{remark}[theorem]{Remark}
\newtheorem{cor}[theorem]{Corollary}
\newtheorem{prop}[theorem]{Proposition}
\newtheorem{defin}[theorem]{Definition}
\newtheorem{result}[theorem]{Result}
\newtheorem{property}[theorem]{Property}

\makeatother
\newcommand{\Prf}{\noindent{\bf Proof}.\quad }
\renewcommand{\labelenumi}{(\alph{enumi})}


\def\B{\mathbf B}
\def\C{\mathbf C}
\def\Z{\mathbf Z}
\def\Q{\mathbf Q}
\def\W{\mathbf W}
\def\a{\mathbf a}
\def\b{\mathbf b}
\def\c{\mathbf c}
\def\d{\mathbf d}
\def\e{\mathbf e}
\def\l{\mathbf l}
\def\v{\mathbf v}
\def\w{\mathbf w}
\def\x{\mathbf x}
\def\y{\mathbf y}
\def\z{\mathbf z}
\def\t{\mathbf t}
\def\cD{\mathcal D}
\def\cC{\mathcal C}
\def\cH{\mathcal H}
\def\cM{{\mathcal M}}
\def\cK{\mathcal K}
\def\cQ{\mathcal Q}
\def\cU{\mathcal U}
\def\cS{\mathcal S}
\def\cT{\mathcal T}
\def\cR{\mathcal R}
\def\cN{\mathcal N}
\def\cA{\mathcal A}
\def\cF{\mathcal F}
\def\cL{\mathcal L}
\def\cP{\mathcal P}
\def\cG{\mathcal G}
\def\cGD{\mathcal GD}

\def\PG{{\rm PG}}
\def\GF{{\rm GF}}

\def\Pg{PG(5,q)}
\def\pg{PG(3,q^2)}
\def\ppg{PG(3,q)}
\def\HH{{\cal H}(2,q^2)}
\def\F{\mathbb F}
\def\Ft{\mathbb F_{q^t}}
\def\P{\mathbb P}
\def\V{\mathbb V}
\def\S{\mathbb S}
\def\bS{\mathbb S}
\def\G{\mathbb G}
\def\E{\mathbb E}
\def\N{\mathbb N}
\def\K{\mathbb K}
\def\D{\mathbb D}
\def\ps@headings{
 \def\@oddhead{\footnotesize\rm\hfill\runningheadodd\hfill\thepage}
 \def\@evenhead{\footnotesize\rm\thepage\hfill\runningheadeven\hfill}
 \def\@oddfoot{}
 \def\@evenfoot{\@oddfoot}
}

\begin{abstract}
In \cite{De} three classes of rank
two presemifields of order $q^{2n}$, with $q$ and $n$ odd, were exhibited, leaving as an open problem the isotopy issue. In \cite{LaMaPoTr2013}, the authors faced with this problem answering the question whether these presemifields are new for $n>3$. In this paper we complete the study solving the case $n=3$.
\end{abstract}


\section{Introduction}

A finite {\it presemifield} $\S=(\S,+,\cdot)$ is a finite algebra
satisfying all the axioms for a {\it skewfield} except the
associativity for the multiplication and the existence of a
multiplicative identity.  If a presemifield admits an identity
element then it is called a {\it semifield}. Semfields have
received a lot of attention in recent years, and we refer
to \cite{LaPo2011} and \cite{Lavrauw2013} for references,
general theory, and definitions, if they are not included here.
A finite presemifield has order a power
of a prime $p$ and such a prime is called the {\it characteristic}
of $\S$. Two presemifields $\S=(\S,+,\cdot)$ and
$\S'=(\S',+,\circ)$ of characteristic $p$ are said to be {\it
isotopic} if there exist three $\F_p$-linear maps $g_1, g_2$ and
$g_3$ from $\S$ into $\S'$ such that $g_1(x\cdot y)=g_2(x) \circ
g_3(y)$ for all $x,y \in \S$. From the isotopy
point of view, presemifields are not more general than semifields;
in fact, by means of the so called Kaplansky trick, it is easy to
see that the isotopy class of a presemifield always contains a
semifield. Such
structures can be also defined for a presemifield as indicated
in \cite[Remark 2.2]{MaPoTr2011} and \cite{MaPo2012}. If two
(pre)semifields are isotopic their dimensions over the nuclei and over
the center are invariant, and we refer to them as the
{\it parameters} of the (pre)semifield. In \cite{De} the author introduced three families of rank $2$ presemifields of order $q^n$, with $q$ and $n$ odd, $2$--dimensional over their left nucleus and $2n$--dimensional over their center.  These presemifields are obtained
starting from a pair of bijective $\F_q$-linear maps of $\F_{q^n}$, satisfying suitable conditions \cite[Theorem 4.1]{De}, and they have been labeled ${\cal D}_A$, ${\cal D}_B$ and ${\cal D}_{AB}$ in \cite{LaMaPoTr2013}.
Also, in \cite[Theorem 4.3]{De}, the author determined their
parameters. This information was insufficient to address the
isotopy issue which, in fact, was leaved as an open problem.
In \cite{LaMaPoTr2013} this problem was solved for $n > 3$,
proving that  presemifields in the families $\cD_A$ and $\cD_{AB}$
are new, i.e. not isotopic to any previously
known semifield; whereas presemifields in the family $\cD_B$ are isotopic to
Generalized Twisted Fields for all $n \geq 3$. The remaining part
of the case $n=3$ is treated here separately; this mainly because
in the relevant case there are many more known examples in the
literature to compare with (\cite{Albert1961P},
\cite{Dickson1906}, \cite{EbMaPoTrComb2009},
\cite{EbMaPoTrEJC2009}, \cite{EbMaPoTrFFA2009},
\cite{JoMaPoTr2011}, \cite{JoMaPoTr2008}, \cite{JoMaPoTr2009}, \cite{Knuth1965},
\cite{LMPT}, \cite{MaPoTr2011}). For $n=3$ the Dempwolff presemifields ${\cal D}_A$ and ${\cal D}_{AB}$
are the algebraic structures ${\cal D}_A=(\F_{q^3}\times \F_{q^3},+,\star_A)$  and ${\cal D}_{AB}=(\F_{q^3}\times \F_{q^3},+,\star_{AB})$, $q$ odd, having multiplications defined as follows
$$(u,v)\star_A (x,y)=(u,v)\begin{pmatrix} x & y \\ A_{a,r}(y) & \xi A_{a,r}(x)
\end{pmatrix},$$
where $A_{a,r}(x)=x^{q^r}-ax^{q^{-r}}$, $r\in\{1,2\}$, such
that $\xi$ is a nonsquare in $\F_q$ and $a\in\F_{q^3}^*$ with $N_{q^3/q}(a) \neq 1$ $^(\footnote{$N_{q^3/q}$ denotes the norm function of $\F_{q^3}$ over $\F_q$}^)$, and
$$(u,v)\star_{AB}(x,y)=(u,v)\begin{pmatrix} x & y \\ A_{b^2,r}(y) & \xi B_{b,-r}(x)
\end{pmatrix},$$
where $A_{b^2,r}(x)=x^{q^r}-b^2x^{q^{-r}}$,
$B_{b,-r}(x)=2H_{b,-r}^{-1}(x)-x$, $r\in\{1,2\}$, such that
$H_{b,-r}(x)=x-bx^{q^{-r}}$ and with $\xi$ a nonsquare in $\F_q$ and
$b\in\F_{q^3}^*$ such that $N_{q^3/q}(b) \neq \pm 1$.

In this paper we prove that  if $N_{q^3/q}(a) =-1$ the presemifields ${\cD}_{A}$ belong, for each given value of $q$ odd, to a unique isotopy class which has been constructed in \cite{EbMaPoTrEJC2009} (Theorem \ref{thm:D_A-notnew}), whereas if $N_{q^3/q}(a) \notin \{-1,1\}$, there are new presemifields in the family $\cD_A$ for each $q$ odd (Theorem \ref{thm:D_A-new}). Regarding presemifields ${\cD}_{AB}$, we prove that they are always new (Theorem \ref{thm:D_AB-new}). In particular when $N_{q^3/q}(b^2) = -1$ presemifields of this sort belong to the class ${\cal F}_3$ introduced in \cite{MaPoTr2007}. The relevant ${\cD}_{AB}$ presemifields provide the first infinite family in such a class whose associated linear set satisfies certain geometric properties; this answers a question posed in \cite{EbMaPoTrEJC2009}. Finally, we prove that presemifields ${\cD}_{A}$ and ${\cD}_{AB}$ are also new up to the Knuth operations and the translation dual operation (Theorem \ref{rem}).

\section{Rank $2$ semifields of order $q^6$ and linear sets}

A pointset  $L$ of a projective space
$\Lambda=PG(r-1,q^n)=PG(V)$, $V=V(r,q^n)$ ($q=p^h$, $p$ prime)  is said
to be an {\em $\F_q$--linear} set of $\Lambda$ of {\it rank} $k$ if it is defined by
the non--zero vectors of a $k$--dimensional $\F_q$--vector subspace $U$ of $V$,
i.e., $$ L=L_U=\{\langle {\bf u}\rangle_{\F_{q^n}}: {\bf u}\in
U\setminus\{{\bf 0}\}\}.\quad\quad (\footnote{In what follows, in
spite of simplicity, we write $\langle {\bf u}\rangle$ to
denote the $\F_{q^n}$--vector space generated by $\bf u$.})$$

Let  $\Omega=PG(W,\F_{q^n})$ be a subspace of $\Lambda$ and let $L_U$
be an $\F_q$-linear set of $\Lambda$. Then $\Omega \cap L_U$ is an
$\F_q$--linear set of $\Omega$ defined by the $\F_q$--vector
subspace $U\cap W$ and, if $dim_{\F_q}(W\cap U)=i$, we say that
$\Omega$ has {\it weight $i$} in $L_U$ and denote this integer by $w_L(\Omega)$. (\footnote{For
further details on linear sets see \cite{Polverino2010} and \cite{LaVa2014}})

Following \cite{BL}, \cite{ML}, we call an $\F_q$--linear set $L_U$ of $\Lambda$ of rank $k$ {\em
scattered} if all of its points have weight 1.  In
\cite[Theorem 4.3]{BL}, it has been proven that a scattered
$\F_q$--linear set of $PG(r-1,q^n)$ has rank at most $rn/2$. A
scattered $\F_q$--linear set $L_U$ of $PG(r-1,q^n)$ of maximum rank
$rn/2$ is called a {\em maximum scattered} linear set. In this paper we will
use the structure of certain type of scattered subspace in the proofs, see below.
Another recent application of scattered subspaces can be found
in \cite{LaShZa2014}.

\bigskip

In \cite{LuMaPoTr-Sub}, generalizing results contained in
\cite{MaPoTr2007}, \cite{LMPT} and  \cite{LV}, a family of maximum
scattered linear sets, called of {\it pseudoregulus type}, is
introduced. Precisely, a scattered $\F_q$--linear set $L_U$ of
$\Lambda=PG(2h-1,q^n)$ of rank $hn$ ($h,n \geq 2$) is of {\it pseoudoregulus type} if
\begin{itemize}
\item[$(i)$] there exist $s_1,s_2,\dots,s_m$ pairwise disjoint lines of $\Lambda$, with $m=\frac{q^{nh}-1}{q^n-1}$, and each of them has weight $n$ in $L_U$;
\item[$(ii)$] there exist exactly
two $(h-1)$--dimensional subspaces $T_1$ and $T_2$ of $\Lambda$
disjoint from $L_U$ such that $T_j\cap s_i\neq \emptyset$ for each
$i=1,\dots,m$ and for each $j=1,2$.
\end{itemize}
The set of lines $\mathcal{P}_{L_U} = \{s_i \colon i=1,\dots,m\}$ is called
the {\it $\F_q$--pseudoregulus} (or simply {\it pseudoregulus}) of
$\Lambda$ associated with $L_U$ and  $T_1$ and $T_2$ are the  {\it
transversal spaces} of $\mathcal{P}_{L_U}$ (or {\it transversal spaces} of
$L_U$). Note that by \cite[Cor. 3.3]{LuMaPoTr-Sub}, if $n>2$ the pseudoregulus $\cP_{L_U}$ associated with $L_U$ and its transversal lines are uniquely determined.

\medskip

Here, we are interested in maximum scattered $\F_q$--linear sets of $\Lambda=PG(3,q^3)$. These linear sets have rank 6 and they are always of pseudoregulus type (see \cite{MaPoTr2007}). Also, the associated pseudoregulus consists of
$q^3+1$ lines, each of weight 3, and its transversal spaces are two disjoint
lines of $\Lambda$. By \cite[Sec. 2]{LMPT}, $\F_q$--linear sets of pseudoregulus type of $PG(3,q^3)$
can be characterized in the following way.

\begin{theorem}[]\label{thm:algebraicpseudoregulus}
Let $t_1=PG(U_1,\F_{q^3})$ and $t_2=PG(U_2,\F_{q^3})$ be two
disjoint lines of $\Lambda=PG(V,\F_{q^3})=PG(3,q^3)$ and let
$\Phi_f$ be a strictly semilinear collineation between $t_1$ and
$t_2$ having as a companion automorphism an element $\sigma\in
Aut(\F_{q^3})$ such that $Fix(\sigma)=\F_q$. Then, for each $\rho
\in \F_{q^3}^*$, the set
$$L_{\rho,f} = \{\langle \underbar{u}+\rho f(\underbar{u}) \rangle \,:\, \underbar{u}\in U_1\setminus\{\underbar 0\}\}$$
is an $\F_q$-linear set of $\Lambda$ of pseudoregulus type whose
associated pseudoregulus is $\mathcal{P}_{L_{\rho,f}}=\{\langle P,
P^{\Phi_f}\rangle\,:\, P\in t_1\}$, with transversal lines $t_1$
and $t_2$.

Conversely, each $\F_q$--linear set of pseudoregulus type of
$\Lambda=PG(3,q^3)$ can be obtained as described above.
\end{theorem}

\medskip

\noindent The semifields that we study in this paper are 6-dimensional $\F_q$-algebras
having at least one nucleus of order $q^3$: {\it rank two semifields
of order $q^{6}$}. The Knuth orbit of such a semifield contains an
isotopism class $[\bS]$ whose left nucleus has size ${q^3}$ and, by means of the multiplicative structure properties,
with the semifield $\bS$ there is associated an $\F_q$-linear set
$L_{\bS}$ of rank $6$ of $PG(3,q^3)$, disjoint from a given hyperbolic quadric $\cQ$. The isotopy class $[\bS]$ corresponds to the
orbit of $L_{\bS}$ under the subgroup $\cG \leq {\mathrm{P\Gamma
O}}^+(4,q^3)$ fixing the reguli of $\cQ$. Furthermore, the image of
$L_{\bS}$ under an element of ${\mathrm{P\Gamma O}}^+(4,q^3)\setminus
\cG$, defines a (pre)semifield which is isotopic to the {\it
transpose} semifield $\bS^t$ of $\bS$ (for more details see
\cite{LaPo2011}).

Let $Tr_{q^3/q}$ denote the trace function of $\F_{q^3}$ over $\F_q$
and let $b(X,Y)$ be the bilinear form associated with $\cQ$. By
field reduction we can use the bilinear form $Tr_{q^3/q}(b(X,Y))$ to
obtain another $\F_q$-linear set, say $L_{\bS}^{\perp}$, of rank $6$
disjoint from $\cQ$. The linear set $L_{\bS}^\perp$ defines a rank
two semifield, say $\bS^{\perp}$, as well; such a semifield is
called the {\it translation dual} of $\bS$ (see e.g.
\cite{LuMaPoTr2008}, or \cite[Section 3]{LaPo2011}). This operation extends the {\it Knuth orbit}
and has been recently generalized in \cite{LuMaPoTr-Sub}.
Another recent extension of the Knuth orbit is contained in \cite{LaShPrep}, using the
so-called BEL-configurations (see \cite{BaEbLa2007}, \cite{Lavrauw2008}, and \cite{Lavrauw2011}).

We will refer to
$\bS^t$ and $\bS^{\perp}$ as the {\it rank two derivatives} of
$\bS$.

\medskip
 As described in \cite{MaPoTr2007}, there are six possible geometric configurations for the linear set $L_{\mathbb{S}}$ and the corresponding
classes of semifields are labeled ${\cal F}_i,$ for $i=0,1,\dots,5$. Semifields belonging to different classes are not isotopic.  Semifields contained in classes ${\cal F}_0$, ${\cal F}_1$ and ${\cal F}_2$ were classified in \cite{MaPoTr2007}.

\medskip

\subsection{Family ${\cal F}_3$}\label{sec:F_3}

\noindent The linear set $L_\bS$ associated with a semifield $\bS$
in class ${\cal F}_3$ has the following structure:
\begin{itemize}
\item [$i)$]  $L_\bS$ contains a unique point of weight greater
than 1 and such a point, say $P$, has weight 2;

\item [$ii)$] $L_\bS$ is not contained in a plane and there exists
a unique plane $\pi$ of $PG(3,q^3)$ of weight 5 in $L_\bS$. Also, since the
weights of $P$ and $\pi$ in $L_\bS$ are 2 and 5, respectively, and
since $L_\bS$ has rank 6, then $P$ is a point of $\pi$.

\item [$iii)$] Any plane of $PG(3,q^3)$, different from $\pi$, has
weight 3 or 4 and if $\pi\ne P^\perp$ (where $\perp$ denotes the
polarity induced by the quadric $\cQ$) then, by $i)$, the point
$\pi^\perp$ has weight 1 or 0, according to $\pi^\perp$ belongs to
$L_\bS$ or not.
\end{itemize}

Since $P$ and $\pi$ are the unique point and the unique plane of
$PG(3,q^3)$ of weight $2$ and 5 in $L_\bS$, respectively, and since the
elements of ${\cal G}$ commute with $\perp$, we have that the
weights of $P$, $\pi$, $P^\perp$ and $\pi^\perp$ in $L_\bS$ are
invariant under isotopisms. Hence, the definition given in
\cite[Section 3]{JoMaPoTr2011} makes sense: a semifield $\bS$
belonging to the class $\cF_3$, with $\pi\ne P^\perp$, is {\it of
type $(i,j)$}, $i\in \{3,4\}$ and $j\in\{0,1\}$, if the weight of
$P^\perp$ in $L_\bS$ is $i$ and the weight of $\pi^\perp$ in
$L_\bS$ is $j$. By \cite[Theorem 3.2]{JoMaPoTr2011}, semifields in
class ${\cal F}_3$ of different types are not isotopic. The
Huang--Johnson semifields $\S_{II}$, $\S_{III}$,
$\S_{IV}$ and $\S_{V}$ (\cite{HuJo1990}) of order $2^6$ belong to the class $\cF_3$. Precisely,
$\S_{II}$ and $\S_{III}$ are of type $(4,1)$, whereas $\S_{IV}$
and $\S_{V}$ are of type $(3,0)$ (see \cite[Prop.
3.5]{JoMaPoTr2011}). Also in \cite{JoMaPoTr2011}, Huang--Johnson
semifields $\S_{II}$ and $\S_{III}$ are extended to new infinite
families of semifields of order $q^6$, existing for every $q$, and all of them are of type $(4,1)$. So far, they are the only known
infinite families of presemifields belonging to the class $\cF_3$.

\medskip

\subsection{Family ${\cal F}_4$}\label{sec:F_4}

\noindent Semifields in the class ${\cal F}_4$ have associated
$\F_q$-linear sets containing a line, say $\ell$, of $PG(3,q^3)$.
This family was further partitioned in \cite{JoMaPoTr2008} into
three subclasses, denoted  ${\cal F}_4^{(a)}$, ${\cal F}_4^{(b)}$
and ${\cal F}_4^{(c)}$. This partition again concerns with the
geometric structure of the associated linear set; precisely, a
semifield $\mathbb{S}$ belongs to the subclass  ${\cal
F}_4^{(a)}$, ${\cal F}_4^{(b)}$ or ${\cal F}_4^{(c)}$ if the polar
line ${\ell}^{\perp}$ of $\ell$ intersects the linear set in $0$,
$1$ or $q+1$ points, respectively. There exist semifields
belonging to family ${\cal F}_4^{(c)}$ for any value of $q$
(\cite{JhJo1992}, \cite{JoMaPoTr2008}, \cite{EbMaPoTrComb2009}),
whereas only two semifields, up to isotopisms, of order $3^6$ are
known to belong to the family ${\cal F}_4^{(b)}$
(\cite{EbMaPoTrFFA2009}). In \cite{EbMaPoTrEJC2009} and \cite{EbMaPoTrFFA2009} it was proven
that, for any $q$ odd, there exists, up to isotopisms, a unique
example in the subclass ${\cal F}_4^{(a)}$, having right nucleus of order $q^2$ and middle nucleus and
center both of order $q$. A representative presemifield of such an
isotopy class is the algebraic structure $(\F_{q^6},+,\star)$ with
multiplication
\begin{equation}\label{form:EMPT}
x\star y = (\alpha+ \beta u + \gamma u^2)x + \gamma b x^{q^3},
\end{equation}
where $u$ and $b$ are given elements of $\F_{q^3}\setminus\F_q$
and $\F_{q^6}$, respectively, such that $u^3=\sigma u+1$,
$\sigma\in\F_q^*$, $b^{q^3+1}=\sigma^2+9u+3\sigma u^2$ and where
$\alpha,\beta,\gamma \in \F_{q^2}$ are uniquely determined
so that $y=\alpha + \beta u+ \gamma (b+u^2)$. It must be
emphasized that semifields in the class ${\cal F}_4^{(a)}$ with either right or
middle nucleus of order $q^2$ are unique up to isotopism; precisely, for each value of $q$, they are isotopic to the presemifield with multiplication (\ref{form:EMPT}) or to its transpose, respectively. Moreover, no such semifields exist when $q$ is even. Also, there could be examples of semifields in class $\cF_4^{(a)}$ for which both
right and middle nuclei have size $q$, although no such examples
have yet been found.

\medskip

\subsection{Family ${\cal F}_5$}\label{sec:F_5}

\noindent Semifields belonging to the family ${\cal F}_5$ are
called of {\em scattered} type; this because the
$\mathbb{F}_q$-linear set associated with such a semifield is
scattered, i.e. it has all points of weight $1$. To any such a
linear set $L$ is associated a geometric object ${\cal P}_L$
called {\it $\mathbb{F}_q$--pseudoregulus}, consisting of $q^3+1$
mutually disjoint lines of $PG(3,q^3)$, each of them intersecting
$L$ in $q^2+q+1$ points, with exactly two transversal lines. From
\cite[Corollary 3.3]{LuMaPoTr-2014}, the pseudoregulus ${\cal
P}_L$ associated with $L$ and its transversal lines are uniquely
determined; hence their position with respect to the quadric $\cQ$
is an isotopism invariant. The known examples of (pre)semifields
belonging to this family are:
\begin{itemize}
\item [$1.$] some Knuth semifields \cite{Knuth1965}: the
associated linear sets have corresponding $\F_q$--pseudo\-regulus
with transversal lines both contained in the quadric $\cQ$
(\cite[Property 4.7]{MaPoTr2007}); \item [$2.$] some Generalized
Twisted Fields \cite{Albert1960}: the associated linear sets have
corresponding $\F_q$--pseudoregulus with transversal lines both
external to the quadric $\cQ$ and pairwise polar with respect to
the polarity defined by $\cQ$ (\cite[Property 4.8]{MaPoTr2007});
\item [$3.$] presemifields constructed in \cite{MaPoTr2011}: the
associated linear sets have corresponding $\F_q$--pseudoregulus
with one transversal line contained in the quadric $\cQ$ and the
other one external to it (\cite[Theorems 4.3, 3.5]{MaPoTr2011});
\item [$4.$] presemifields constructed in \cite{LMPT}: the
associated linear sets have corresponding $\F_q$--pseudoregulus
with transversal lines $t_1$ and $t_2$ both external to the
quadric $\cQ$ and not pairwise polar. Moreover, in these cases,
$t_1^\perp \cap t_2 = \emptyset$ (\cite[Section 5]{LMPT}).
\end{itemize}

In \cite[Section 3]{LMPT}, a canonical form for presemifields in family $\cF_5$ whose associated $\F_q$--pseudoregulus has at least one of the transversals external to the quadric $\cQ$ has been determined. As we will show later, some $\cD_A$ presemifields belong to this class and so, in order to compare them with the
last examples of the previous list, we need to further investigate this class of presemifields.

\medskip

\noindent The above mentioned canonical form has been written studying the associated linear sets in the projective space $PG(V)$, where $V=\F_{q^6}\times\F_{q^6}$ is considered as a vector space over $\F_{q^3}$. These linear sets are disjoint from the hyperbolic quadric $\bar\cQ$ of $PG(V)$ with equation $X^{q^3+1}-Y^{q^3+1}=0$ and, using the action of the collineation group fixing the reguli of such a quadric, w.l.o.g. we may assume
 that one of the tranversals of the associated $\F_q$--pseudoregulus is the external line $\ell:=\{\langle ( y,0)\rangle~:~y \in \F_{q^6}^*\}$. In \cite[Theorem 3.3]{LMPT} a canonical form for such presemifields has been exhibited and they have been denoted by ${\mathbb{S}}(\lambda,\mu,\alpha,\beta,\sigma)$. Moreover, presemifields in such a family having right and middle nuclei of orders $q^2$ and $q$, respectively, belong to the subfamily ${\mathbb{S}}(\lambda,0,\alpha,0,\sigma)$. Precisely, by \cite[Theorem 4.1]{LMPT} and by \cite[Theorem 3.4]{MaPoTr2011}, these presemifields are isotopic to the algebraic structure $(\F_{q^6},+,\star)$, whose multiplication rule is given by
\begin{equation}\label{canonicalform}
x\star y = (\lambda y + \alpha y^{\sigma})x + yx^{q^3},
\end{equation}
where $\lambda,\alpha\in\F_{q^6}$, $\alpha\ne 0$, $\sigma\in\{q^2,q^4\}$ such that
$$N_{q^6/q^3}(y) \neq N_{q^6/q^3}(\lambda y +\alpha y^{\sigma})\quad\quad\mbox{for each $y\in \F_{q^6}^*$.}$$ Also,
the element $\lambda$ can be taken up to its norm, i.e., if
$N_{q^6/q^3}(\lambda)=N_{q^6/q^3}(\lambda')$, then the two presemifields ${\mathbb{S}}(\lambda,0,\alpha,0,\sigma)$ and ${\mathbb{S}}(\lambda',0,\alpha,0,\sigma)$ are isotopic (\cite[Theorem 3.3]{LMPT}).

\smallskip

\noindent The linear set associated with ${\mathbb{S}}(\lambda,0,\alpha,0,\sigma)$ in the projective space $PG(V)$ is
$$L(\lambda,0,\alpha,0,\sigma):= \{\langle (\lambda y+\alpha y^{\sigma}, y)\rangle  \colon
y\in \F_{q^6}^*\}.$$
It is of pseudoregulus type and, by \cite[Section 3]{LMPT}, the corresponding transversals are
\begin{displaymath}
\begin{array}{rl}
\ell & :=\{\langle ( y,0)\rangle~:~y \in \F_{q^6}^*\}\\
\ell_{\lambda} & :=\{\langle ( \lambda y,y)\rangle~:~y \in \F_{q^6}^*\}.
\end{array}
\end{displaymath}

\bigskip

In the following we will give a necessary condition assuring that two presemifields of type ${\mathbb{S}}(\lambda,0,\alpha,0,\sigma)$ are isotopic, which will be useful in the sequel.

\begin{lemma}\label{lm:NecessaryCond}
If two presemifields ${\mathbb{S}}(\lambda,0,\alpha,0,\sigma)$ and ${\mathbb{S}}(\bar{\lambda},0,\bar{\alpha},0,\bar{\sigma})$ are isotopic, then $N_{q^6/q^3}(\bar\lambda^{\tau})= N_{q^6/q^3}(\lambda)$, for some $\tau \in Aut(\F_{q^6})$.
\end{lemma}
\begin{proof}
If ${\mathbb{S}}(\lambda,0,\alpha,0,\sigma)$ and ${\mathbb{S}}(\bar{\lambda},0,\bar{\alpha},0,\bar{\sigma})$ are isotopic then the associated linear sets are isomorphic under the action of the collineation group $\cG$ fixing the reguli of $\bar\cQ$. The elements of the group $\cG$ (whose linear part has been described in \cite[Section 3]{LMPT}) are
\begin{equation}\label{lcollineation}
\langle x,y\rangle \mapsto \langle AC^\tau x^\tau+BD^{\tau q^3}x^{\tau q^3}+AD^{^\tau q^3}y^\tau+BC^\tau y^{\tau q^3},AD^\tau x^\tau+BC^{\tau q^3}x^{\tau q^3}+AC^{\tau q^3}y^\tau+BD^\tau y^{\tau q^3}\rangle
\end{equation}
where $A,B,C$ and $D$ are elements of $\F_{q^6}$, with $A^{q^3+1} \neq B^{q^3+1}$ and $C^{q^3+1} \neq D^{q^3+1}$, and $\tau$ is an automorphism of $\F_{q^6}$. Since the transversals of the $\F_q$--pseudoregulus associated with such presemifields are uniquely determined (\cite[Corollary 3.3]{LuMaPoTr-2014}), the pair $\{\ell,\ell_{\bar\lambda}\}$ must be mapped, under the action of the group $\cG$, to the pair $\{\ell,\ell_{\lambda}\}$.

We first suppose that there exists an element $\phi$ of $\cG$ such that $\phi(\ell)=\ell$ and $\phi(\ell_{\bar\lambda})=\ell_\lambda$. In this case $\phi$ belongs to the stabilizer $\cG_\ell$ of $\ell$ in $\cG$, whose elements, by (\ref{lcollineation}), are
\begin{equation}\label{stabilizer}
\langle x,y\rangle \mapsto \langle BD^{\tau q^3}x^{\tau q^3},BD^\tau y^{\tau q^3}\rangle
\end{equation}
and
\begin{equation}\label{stabilizer1}
\langle x,y\rangle \mapsto \langle AC^\tau x^\tau,AC^{\tau q^3}y^\tau\rangle,
\end{equation}
where $A,B,C,D\in\F_{q^6}^*$ and $\tau$ is an automorphism of $\F_{q^6}$. In both cases, requiring $\phi(\ell_{\bar\lambda})=\ell_\lambda$, we get $N_{q^6/q^3}(\bar\lambda^{\tau})= N_{q^6/q^3}(\lambda)$.

Now, suppose there exists an element $\phi$ of $\cG$ such that
\begin{equation}\label{form1}
\phi(\ell)=\ell_{\lambda}
\end{equation} and
\begin{equation}\label{form2}
\phi(\ell_{\bar\lambda})=\ell.
\end{equation} By substituting $C^\tau$ by $C$ and $D^\tau$ by $D$, from (\ref{form1}), we get
\begin{equation}\label{form3}
A(C-\lambda D)=0,
\end{equation}
\begin{equation}\label{form4}
B(D^{q^3}-\lambda C^{q^3})=0
\end{equation}
and from (\ref{form2}) we have
\begin{equation}\label{form5}
A(D\bar\lambda^\tau+C^{q^3})=0,
\end{equation}
\begin{equation}\label{form6}
B(C^{q^3}\bar\lambda^{\tau q^3}+D)=0.
\end{equation}
If $A\ne 0$ and $B\ne 0$, combining (\ref{form3}) and (\ref{form4}) we get $C^{q^3+1}=D^{q^3+1}$, a contradiction. Hence, either $A=0$ or $B=0$. If $A=0$, by (\ref{form4}) and (\ref{form6}) we have
$$D^{q^3}-\lambda C^{q^3}=C^{q^3}\bar\lambda^{\tau q^3}+D=0,$$ which again implies $N_{q^6/q^3}(\bar\lambda^{\tau})= N_{q^6/q^3}(\lambda)$. The same condition  is obtained when $B=0$ and this completes the proof.
\end{proof}

\section{The 6-dimensional $\cD_{A}$ and $\cD_{AB}$ presemifields}

In this section we study the isotopy issue for the presemifields $\cD_A$ and $\cD_{AB}$ with dimension 6 over their center, by comparing them with the known examples mentioned in the previous section.

\bigskip

\subsection{$\cD_{A}$ presemifields}

A $\cD_A$ presemifield of order $q^6$, $q$ odd, is the algebraic structure $(\F_{q^3}\times\F_{q^3},+,\star_A)$, with multiplication
$$(u,v)\star_A(x,y)=(u,v)\begin{pmatrix} x & y \\ A_{a,r}(y) & \xi A_{a,r}(x)
\end{pmatrix},$$
where $A_{a,r}(x)=x^{q^r}-ax^{q^{-r}}$, $r\in\{1,2\}$, such
that $\xi$ is a nonsquare in $\F_q$ and $a\in\F_{q^3}^*$ with $N_{q^3/q}(a) \neq 1$. Note that, since $\cD_A$ has no zero divisors, for each $(x,y)\ne (0,0)$ the corresponding matrix $\begin{pmatrix} x & y \\ A_{a,r}(y) & \xi A_{a,r}(x)
\end{pmatrix}$ is invertible, i.e. $\xi x A_{a,r}(x)\ne y A_{a,r}(y)$.

\noindent The associated $\F_q$--linear set in $\P=PG(3,q^3)$
$$L_{\cD_A}=\{\langle (x,y,A_{a,r}(y),\xi A_{a,r}(x))\rangle\, : \, x,y \in \F_{q^3}, (x,y) \neq (0,0)\},$$
is disjoint from the hyperbolic quadric $\cQ=Q^+(3,q^3):\,X_0X_3-X_1X_2=0$.

\medskip

Also, by \cite[Theorem 4.3]{De}, a $\cD_A$ presemifield has left nucleus of size $q^3$, right nucleus of size $q^2$, middle nucleus and center both of size $q$.

\bigskip

\noindent In the following we prove that the presemifields $\cD_A$ are not new when
$N_{q^3/q}(a)=-1$ and, for each given $q$, all
of them belong to a unique isotopy class, whereas, if $N_{q^3/q}(a)\notin\{-1,1\}$, the family $\cD_A$ contains, for each $q\geq 5$ odd, new presemifields.

\begin{theorem}\label{thm:D_A-notnew}
If $N_{q^3/q}(a)=-1$, then presemifields $\cD_A$ are all isotopic
to the presemifield whose multiplication is given in
(\ref{form:EMPT}).
\end{theorem}
\begin{proof}
Since $N_{q^3/q}(a)=-1$, by \cite[Theorem 4.5
$(iii)$]{LaMaPoTr2013}, the linear set $L_{\cD_A}$ is not
scattered. Moreover, it contains $q+1$ points of weight $2$ and
these points belong to a line which is contained in $L_{\cD_A}$
and hence, $\cD_A$ belongs to family ${\cal F}_4$. Also, since the
right nucleus has order $q^2$ and the middle nucleus and center
have both size $q$, applying \cite[Theorem 3.5]{EbMaPoTrFFA2009},
we can conclude that the presemifield $\cD_A$ belongs to the
family ${\cal F}_4^{(a)}$. Now, in (\cite[Theorem
2.7]{EbMaPoTrEJC2009}) the authors construct an infinite family of
odd order presemifields belonging to such a subclass and having
right nucleus of size $q^2$, also proving that, up to isotopy,
this family is unique (\cite[Corollary 3.2]{EbMaPoTrEJC2009}). So,
taking also into account \cite[Theorem 1.1]{MaPoTr2011}, the
presemifields $\cD_A$, with $N_{q^3/q}(a)=-1$, are all isotopic to
the presemifield whose multiplication is defined as in
(\ref{form:EMPT}).
\end{proof}

\begin{theorem}\label{thm:D_A-new}
If $N_{q^3/q}(a)\notin\{-1,1\}$, the family $\cD_A$ contains new presemifields.
\end{theorem}
\begin{proof}
If $N_{q^3/q}(a)\notin\{-1,1\}$, by \cite[Theorem 4.5
$(i)$]{LaMaPoTr2013} the linear set $L_{\cD_A}$ is a maximum
scattered $\F_q$--linear set of $\P=PG(3,q^3)$ of pseudoregulus
type. Let $t_1$ and $t_2$ be the lines of $\P$ with equations
$$t_1:\ \left\{
\begin{array}{l}
X_2=-a^{q^r+1}X_1\\
X_3=-\xi a^{q^r+1}X_0
\end{array}
\right.\quad\quad\quad t_2:\ \left\{
\begin{array}{l}
X_1=a^{q^{-r}}X_2\\
X_0=\frac{a^{q^{-r}}}{\xi}X_3.
\end{array}
\right.$$ Since $q$ is odd, $\xi$ is a nonsquare in $\F_q$ and
$N_{q^3/q}(a)\notin\{-1,1\}$, these two lines are disjoint, both external
to the quadric $\cQ$ and $t_1^\perp\cap t_2=\emptyset$.

\noindent Let $\Phi_f$ be the collineation between $t_1$ and $t_2$
induced by the strictly semilinear map
$$f:\ (x,y,-a^{q^r+1}y,-\xi a^{q^r+1}x)\mapsto\ (a^{q^{-r}}x^{q^r}, a^{q^{-r}}y^{q^r},y^{q^r},\xi x^{q^r}),$$ taking into account that $x^{q^{-2r}}=x^{q^r}$ for each $x\in\F_{q^3}$, we get
$$\{\langle (x,y,-a^{q^r+1}y,-\xi a^{q^r+1}x)+f((x,y,-a^{q^r+1}y,-\xi a^{q^r+1}x)) \rangle:\,x,y \in \F_{q^3}, (x,y) \neq (0,0)\}$$
$$= \{\langle (x+a^{q^{-r}}x^{q^r},y+a^{q^{-r}}y^{q^r},-a^{q^r+1}y+y^{q^r},-\xi a^{q^r+1}x+\xi x^{q^r}) \rangle \,:\,x,y \in \F_{q^3}, (x,y) \neq (0,0)\}$$
$$=\{\langle (z,w,A_{a,r}(w),\xi A_{a,r}(z))\rangle\, :\,z,w \in \F_{q^3}, (z,w) \neq (0,0)\}=L_{\cD_A}.$$
By Theorem \ref{thm:algebraicpseudoregulus}, the lines $t_1$ and $t_2$ are the two
transversals of the $\F_q$--pseudoregulus associated with
$L_{\cD_A}$.  Hence, if $N_{q^3/q}(a)\notin\{-1,1\}$, the presemifields $\cD_A$ belong to the family $\cF_5$ and the associated $\F_q$--pseudoreguli have both the
transversals external to the hyperbolic quadric $\cQ$. Then, by \cite[Theorem 3.2]{LMPT} they belong to the isotopy class of some
semifield ${\mathbb{S}}(\lambda,\mu,\alpha,\beta,\sigma)$. Moreover, since the  presemifields $\cD_A$ have right nucleus of size $q^2$, middle nucleus
and center both of size $q$, then, by \cite[Theorem 4.1]{LMPT} and \cite[Theorems 3.4
and 3.5]{MaPoTr2011}, they belong to the isotopism class of some
semifield ${\mathbb{S}}(\lambda,0,\alpha,0,\sigma)$, with
$\lambda,\alpha\in\F_{q^6}^*$, $N_{q^6/q^3}(\lambda)\ne 1$ and
$\sigma\in\{q^2,q^4\}$, for which a canonical form has been given in (\ref{canonicalform}). So, to compare the presemifields $\cD_A$ with the known examples of odd order presemifields of type ${\mathbb{S}}(\lambda,0,\alpha,0,\sigma)$ constructed in \cite[Section 5]{LMPT}, we will write the linear set $L_{\cD_A}$ in the relevant canonical form. In order to do this, we apply an
isomorphism $\varphi$ from $\P=PG(3,q^3)$ to $PG(V)$ such that the
quadric $\cQ:X_0X_3-X_1X_2=0$ of $\P$ is mapped onto the quadric
$\bar\cQ~:~X^{q^3+1}-Y^{q^3+1}=0$ of $PG(V)$, and then move, using the action of the group $\cG$ fixing the reguli of $\bar\cQ$, the transversal $t_1$ to $\ell$. Let
\begin{displaymath}
\varphi~:~\langle (x_0,x_1,x_2,x_3)\rangle\in\P ~\mapsto~\langle (x_0+x_3+(x_1+\xi x_2)\omega, x_0-x_3 + (x_1-\xi x_2)\omega)\rangle\in PG(V),
\end{displaymath}
where $\omega \in \F_{q^2}\setminus \F_{q}$, with $\omega^2=1/\xi$, and let $\Psi$ be the collineation of $\cal G$ induced in $PG(V)$ by the nonsingular
matrix
\begin{displaymath}
\left (
\begin{array}{ll}
1-\xi a^{q^r+1} & -1-\xi a^{q^r+1}\\
-1-\xi a^{q^r+1} & 1-\xi a^{q^r+1}\\
\end{array}
\right ),
\end{displaymath}
we have
\begin{displaymath}
\begin{array}{rl}
t_1^{\varphi \Psi}& =\{ \langle (y,0) \rangle ~:~ y \in \F_{q^6}^*\}=\ell,\\
t_2^{\varphi \Psi} & =\{ \langle (2\xi(1-N_{q^3/q}(a))y,-2\xi (1+N_{q^3/q}(a))y) \rangle ~:~ y \in \F_{q^6}^*\},\\
 & =\{ \langle (\bar\lambda y,y) \rangle ~:~ y \in \F_{q^6}^*\}=\ell_{\bar \lambda},\\
\end{array}
\end{displaymath}
where
\begin{equation}\label{form:lambda}
\bar\lambda=\frac{N_{q^3/q}(a)-1}{N_{q^3/q}(a)+1}.
\end{equation}

\noindent Note that, as $a$ runs over $\F_{q^3}^*$ with the condition $N_{q^3/q}(a)\notin\{-1,1\}$ then $\bar\lambda$ covers $\F_q\setminus\{0,-1,1\}$.

Hence,  $L_{\cD_A}^{\varphi\psi}$  is an $\F_q$--scattered linear set of pseudoregulus type with transversals $\ell$ and $\ell_{\bar\lambda}$ and the corresponding presemifield isotopic to $\cD_A$ is of type ${\mathbb{S}}(\bar\lambda,0,\bar\alpha,0,\bar\sigma)$, for some $\bar\alpha\in\F_{q^6}^*$, $\bar\sigma\in\{q^2,q^4\}$ and with $\bar\lambda$ as in (\ref{form:lambda}).

So far, the only known examples of semifields of type ${\mathbb{S}}(\lambda,0,\alpha,0,\sigma)$ have been constructed in \cite[Section 5]{LMPT}. This means that, in order to prove that
the family $\cD_A$ contains new presemifields, it has to be compared with the examples of odd order
constructed in \cite{LMPT} and listed below:
\begin{itemize}
\item[$\S_1$:] ${\mathbb{S}}(\lambda_1,0,\lambda_1,0,q^2)$, $q\equiv 1(mod\,6)$, with $N_{q^6/q^3}(\lambda_1)^{-1}=2(1-\eta)$, where $\eta$ is a primitive
6-th root of unity over $\F_q$ (\cite[Theorem 5.3]{LMPT});
\item[$\S_2$:] ${\mathbb{S}}(\lambda_2,0,\lambda_2,0,q^2)$, $q\equiv 0(mod\,3)$, with $N_{q^6/q^3}(\lambda_2)^{-1}=-(u+u^2)$, where $u\in
\F_{q^3}\setminus \F_q$ such that $u^3=u+1$ (\cite[Theorem 5.8]{LMPT}).
\end{itemize}

\noindent Suppose that a presemifield $\cD_A$ (which is, up to isotopy, of type ${\mathbb{S}}(\bar\lambda,0,\bar\alpha,0,\bar\sigma)$) is isotopic to one of the presemifields ${\mathbb{S}}_i={\mathbb{S}}(\lambda_i,0,\lambda_i,0,q^2)$, $i\in\{1,2\}$. Then, by Lemma \ref{lm:NecessaryCond}, we get
\begin{equation}\label{form:equazione}
N_{q^6/q^3}(\bar\lambda^\tau)=N_{q^6/q^3}({\lambda_i}),
\end{equation}
for some automorphism $\tau$ of $\F_{q^6}$.
By (\ref{form:lambda}), $\bar \lambda\in \F_q$ and, by (\ref{form:equazione}), we get
$\bar\lambda^{2\tau}=N_{q^6/q^3}(\lambda_i)\in \F_q$. Since $N_{q^6/q^3}(\lambda_2)\notin \F_q$, a presemifield $\cD_A$ can only be isotopic to a presemifield $\S_1$. Now, we first observe that if $\eta$ is a primitive 6-th root of unity over $\F_q$, then $\eta^\tau$ is either $\eta$ or $\eta^{-1}$. Since $\eta$ and $\eta^{-1}$ are not necessarily conjugated under an automorphism of $\F_{q^6}$, from Lemma \ref{lm:NecessaryCond}, it follows
that there are, for each $q\equiv 1(mod\,6)$, at most 2 non--isotopic presemifields of type $\S_1$.

On the other hand if two presemifields $\cD_A$ of type ${\mathbb{S}}(\bar\lambda,0,\bar\alpha,0,\bar\sigma)$ and ${\mathbb{S}}(\bar\lambda',0,\bar\alpha',0,\bar\sigma')$, respectively, are isotopic, by Lemma \ref{lm:NecessaryCond}, we get $N_{q^6/q^3}(\bar\lambda^\tau)=N_{q^6/q^3}(\bar\lambda')$, i.e. by (\ref{form:lambda}) $\bar\lambda^{2\tau}=\bar\lambda'^2$, where $\bar\lambda,\bar\lambda'\in\F_q\setminus\{0,1,-1\}$ and $\tau$ can be considered as an automorphism of $\F_q$, with $q=p^h$. Since the number of nonzero squares of $\F_q$ different from 1 is $\frac {q-3}2$ and since each of them has at most $h$ conjugates in $\F_q$, the number of non--isotopic presemifields $\cD_A$, of given order $q^6$, is at least $\frac{q-3}{2h}$. Now, if $q>7$ then $\frac{q-3}{2h}>2$, i.e. for $q>7$, there are always presemifields in the family $\cD_A$ which are not isotopic to any previously known semifield. Finally, if $q=7$ and $\cD_A$ is isotopic to $\S_1$, then by (\ref{form:equazione}) $$\bar\lambda^{2\tau}=N_{q^6/q^3}(\lambda_1)=\frac 1{2(1-\eta)},$$ which means, since $N_{q^3/q}(a)\notin\{0,1,-1\}$ and taking (\ref{form:lambda}) into account, that $\frac 1{2(1-\eta)}$ is a nonzero square of $\F_q$ different from 1. But this is not the case since $\eta\in\{3,5\}$.

This completes the proof.
\end{proof}

\subsection*{$\cD_{AB}$ presemifields}

A $\cD_{AB}$ presemifield of order $q^6$, $q$ odd, is the
algebraic structure $(\F_{q^3}\times\F_{q^3},+,\star_{AB})$, with
multiplication
$$(u,v)\star_{AB}(x,y)=(u,v)\begin{pmatrix} x & y \\ A_{b^2,r}(y) & \xi B_{b,-r}(x)
\end{pmatrix},$$
where $A_{b^2,r}(y)=y^{q^r}-b^2y^{q^{-r}}$,
$B_{b,-r}(x)=2H_{b,-r}^{-1}(x)-x$, $r\in\{1,2\}$, such that
$H_{b,-r}(x)=x-bx^{q^{-r}}$ and with $\xi$ a nonsquare in $\F_q$ and
$b\in\F_{q^3}^*$ with $N_{q^3/q}(b) \neq \pm 1$. By \cite[Theorem
4.3]{De}, a $\cD_{AB}$ presemifield has left nucleus of size $q^3$,
right nucleus, middle nucleus and center all of size $q$. We can
prove the following

\begin{theorem}\label{thm:D_AB-new}
The family of presemifields $\cD_{AB}$ is new for each value of $q$. Also,
\begin{itemize}
\item[i)] if $N_{q^3/q}(b^2)=-1$ then $\cD_{AB}$ is contained in the class ${\cal F}_3$ and it is of type $(3,0)$;
\item[ii)] if $N_{q^3/q}(b^2)\neq -1$ then $\cD_{AB}$ is contained in the class ${\cal F}_5$ and the transversals $t_1$ and $t_2$ of the associated pseudoregulus satisfy the condition $|t_1^\perp\cap t_2|=1$.
\end{itemize}
\end{theorem}
\begin{proof}
The $\F_q$--linear set
$\P=PG(3,q^3)$ associated with a presemifield $\cD_{AB}$ is
$$L_{\cD_{AB}}=\{\langle (x,y,y^{q^r}-b^2y^{q^{-r}},\xi (2H_{b,-r}^{-1}(x)-x))\rangle\, : \, x,y \in \F_{q^3}, (x,y) \neq (0,0)\}$$
and it is disjoint from the hyperbolic quadric $\cQ:
X_0X_3-X_1X_2=0$. Also, the two lines $s:X_1=X_2=0$ and
$s^{\perp}: X_0=X_3=0$ of $\P$ have both weight 3 in $L_{\cD_{AB}}$. Since $N_{q^3/q}(b) \neq 1$, the map $H_{b,-r}(x)$ is invertible and
hence $L_{\cD_{AB}}$ can be rewritten in the following fashion:
$$L_{\cD_{AB}}=\{\langle (x-bx^{q^{-r}},y,y^{q^r}-b^2y^{q^{-r}},\xi(x+bx^{q^{-r}}))\rangle \, : \, x,y \in \F_{q^3}, (x,y) \neq (0,0)\}.$$

$i)$\quad If $N_{q^3/q}(b^2)=-1$ (which implies $q\equiv 1(mod\,4)$), by
\cite[Theorem 4.8 $(iii)$]{LMPT}, $L_{\cD_{AB}}$ is
not scattered and it contains a unique point, say $P$, of weight
2. Hence, by \cite[Section $4$]{MaPoTr2007}, the presemifields $\cD_{AB}$
belong to the family $\cF_3$. Moreover, by the proofs of
\cite[Lemma 4.4, Theorems 4.5 and 4.8]{LMPT}, such a point is on
the line $s^\perp$ and it has coordinates $P\equiv (0,\bar y,\bar
y^{q^r}-b^2\bar y^{q^{-r}},0)$, where $\bar y$ is a solution of
the equation
$$x^{q^{2r}-1}=-\frac{b^{2q^r}}{(\lambda-\lambda^{q^r})^{q^r-1}},$$
with $\lambda$ a given element of $\F_{q^3}\setminus\F_q$. Since
the line $s$ has weight 3 in $L_{\cD_{AB}}$, the plane
$\pi=\langle P,s\rangle$ is the unique plane of $\P$ of weight 5
in $L_{\cD_{AB}}$ (by the first two properties of Section \ref{sec:F_3}). Also, $P^\perp$ is the plane with equation $\bar
yX_2+(\bar y^{q^r}-b^2\bar y^{q^{-r}})X_1=0$ and, taking into
account that $N_{q^3/q}(b^2)=-1$, straightforward computations
show that $$L_{\cD_{AB}}\cap P^\perp=L_{\cD_{AB}}\cap s.$$ It
follows that $w_{L_{\cD_{AB}}}(P^\perp)=w_{L_{\cD_{AB}}}(s)=3$ and,
since $\pi^\perp=P^\perp\cap s^\perp$ (and hence $\pi^\perp\notin
s$), we get $w_{L_{\cD_{AB}}}(\pi^\perp)=0$. This implies that,
when $N_{q^3/q}(b^2)=-1$, each presemifield $\cD_{AB}$ belongs to
the family $\cF_3$ and it is of type $(3,0)$. The only known
presemifields in the literature of order $q^6$, with $q$ odd,
belonging to the family $\mathcal{F}_3$ are those described in
\cite[Proposition 3.8 \,and\,Theorem 3.9]{JoMaPoTr2011}. However all
presemifields in this infinite family are of type $(4,1)$. This
implies that when $N_{q^3/q}(b^2)=-1$, the presemifields of the family
$\cD_{AB}$ are all new.

\medskip

$ii)$\quad Let now $N_{q^3/q}(b^2)\ne -1$. Then, by \cite[Theorem 4.8
$(i)$]{LMPT}, the linear set $L_{\cD_{AB}}$ is a maximum scattered
linear set of $\P$ of pseudoregulus type, and hence the
presemifields $\cD_{AB}$ belong to the family $\cF_5$. Let $t_1$ and $t_2$ be the lines of $\P$ with equations
$$t_1:\ \left\{
\begin{array}{l}
X_2=b^{-2q^{-r}}X_1\\
X_3=\xi X_0
\end{array}
\right.\quad\quad\quad t_2:\ \left\{
\begin{array}{l}
b^{-2q^{r}}X_2=-b^{2}X_1\\
X_3=-\xi X_0.
\end{array}
\right.$$
Since $q$ is odd, $\xi$ is a nonsquare in $\F_q$ and
$N_{q^3/q}(b^2)\notin\{-1,1\}$, these two lines are disjoint, both
external to the quadric $\cQ$ and $t_1^\perp\cap t_2$ is the point
$\langle(1,0,0,-\xi)\rangle$.

Let  $\Phi_f$ be the collineation between $t_1$ and $t_2$ induced by the strictly semilinear map
$$f:\ (x,y,b^{-2q^{-r}}y,\xi x)\mapsto\ (-bx^{q^{-r}},b^{-2q^{r}}y^{q^{-r}},-b^2y^{q^{-r}},\xi b x^{q^{-r}}),$$
direct computations show that
$$\{\langle
(x,y,b^{-2q^{-r}}y,\xi x)+f((x,y,b^{-2q^{-r}}y,\xi x)) \rangle
\,:\, x,y \in \F_{q^3},\  (x,y)\neq (0,0)\}\quad\quad$$
$$= \{\langle
(x-bx^{q^{-r}},y+b^{-2q^{r}}y^{q^{-r}},b^{-2q^{-r}}y-b^2y^{q^{-r}},\xi (x+b x^{q^{-r}})) \rangle
\,:\, x,y \in \F_{q^3},\  (x,y)\neq (0,0)\}$$
$$= \{\langle
(x-bx^{q^{-r}},z,z^{q^r}-b^2z^{q^{-r}},\xi (x+b x^{q^{-r}})) \rangle
\,:\, x,z \in \F_{q^3},\  (x,z)\neq (0,0)\}=L_{\cD_{AB}}.$$

Then by Theorem \ref{thm:algebraicpseudoregulus}, the lines $t_1$ and $t_2$ are the two
transversals of the $\F_q$--pseudoregulus associated with
$L_{\cD_{AB}}$. Hence, when $N_{q^3/q}(b^2)\ne -1$,
presemifields $\cD_{AB}$ belong to the class $\mathcal{F}_5$
and provide the first examples in the literature of presemifields
in class ${\cal F}_5$ such that the transversals of the associated
$\F_q$--pseudoregulus,  satisfy the geometric condition $|t_1^\perp\cap t_2|=1$. Since the size of this intersection is an isotopism invariant, taking into account the known examples of presemifields in the family $\cF_5$ listed in Section 2, we can conclude that the presemifields $\cD_{AB}$, with $N_{q^3/q}(b^2)\ne -1$, are all new.
\end{proof}

Finally, we want to point out the following result.

\begin{theorem}\label{rem}
The presemifields $\cD_{AB}$ and each new presemifield $\cD_{A}$ are not the rank two derivatives of any known presemifield.
\end{theorem}

\begin{proof}
The geometric properties defining presemifields of type $(3,0)$ in family $\cF_3$ are invariant under the transpose and the translation dual operations (\cite[Theorems 3.2 and 3.3]{JoMaPoTr2011}). Moreover, the geometric configurations (with respect to the quadric $\cQ$) of the transversals of the pseudoregulus associated with one of the known semifields in class $\cF_5$ (listed in Section 2) are invariant under the above mentioned operations. This means that 6--dimensional $\cD_{AB}$ presemifields are isotopic neither to the transpose nor to the translation dual of any know presemifield.

The transpose of each presemifield $\cD_A$ has middle and right nuclei of sizes $q^2$ and $q$, respectively (\cite{Maduram1975}). Since the family of presemifields of type ${\mathbb{S}}(\lambda,\mu,\alpha,\beta,\sigma)$ is closed under the transpose operation (\cite[Theorem 3.3]{LMPT}) and since we have proven that the family $\cD_{A}$ is contained, up to isotopy, in such a family, the transpose $\cD_A^t$ of a new presemifield  $\cD_A$ could only be isotopic to one of the odd order examples constructed in \cite[Thorems 5.3 and 5.8]{LMPT} having the same parameters, i.e. it could only be isotopic to either $\bS_1^t$ or $\bS_2^t$, in which case $\cD_A$ would be isotopic to either $\S_1$ or $\S_2$. But in both cases Theorem \ref{thm:D_A-new} proves that this is not the case. Finally, note that the family of presemifields of type ${\mathbb{S}}(\lambda,\mu,\alpha,\beta,\sigma)$ is closed under the translation dual operation (\cite[Theorem 3.3]{LMPT}) and such an operation leaves invariant the parameters of a presemifield (\cite[Theorem 5.3]{LuMaPoTr2008}). Hence, we only have to compare each new presemifield $\cD_{A}$ with the translation duals of $\bS_1$ and $\bS_2$. Since the translation dual of $\bS_i={\mathbb{S}}(\lambda_i,0,\lambda_i,0,q^2)$ is, up to isotopy, the semifield $\bS_i^\perp={\mathbb{S}}(\lambda_i^{q^3},0,\lambda_i^q,0,q^4)$ and this is isotopic to ${\mathbb{S}}(\lambda_i,0,\lambda_i^q,0,q^4)$ (\cite[Theorem 3.3]{LMPT}), the same arguments as in the proof of Theorem \ref{thm:D_A-new} prove that each new presemifield $\cD_A$ is also new with respect to the translation dual operation.
\end{proof}

\bigskip

\noindent  Michel Lavrauw\\
Department of Management and Engineering,\\
Universit\`a di Padova,\\
I--\,36100 Vicenza, Italy\\
{\em michel.lavrauw@unipd.it}

\medskip

\noindent Giuseppe Marino and Olga Polverino\\
Dipartimento di Matematica e Fisica,\\
 Seconda Universit\`a degli Studi
di Napoli,\\
I--\,81100 Caserta, Italy\\
{\em giuseppe.marino@unina2.it}, {\em olga.polverino@unina2.it}

\medskip
\noindent  Rocco Trombetti\\
Dipartimento di Matematica e Applicazioni "R. Caccioppoli",\\
Universit\`a degli Studi di Napoli ``Federico II'',\\
 I--\,80126 Napoli, Italy\\
{\em rtrombet@unina.it}

\end{document}